\documentclass{article}
\usepackage{amssymb}

\usepackage{amsmath}

\newtheorem{theorem}{Theorem}

\newtheorem{condition}[theorem]{Condition}

\newtheorem{remark}[theorem]{Remark}

\newenvironment{proof}[1][Proof]{\noindent\textbf{#1.} }{\ \rule{0.5em}{0.5em}}
\input{tcilatex}

\begin{document}

\title{On The Dependence Structure of Wavelet Coefficients for Spherical
Random Fields\thanks{%
We are very grateful to a referee and an Associate Editor for their very
constructive remarks. The research of Xiaohong Lan was supported by the
Equal Opportunity Committee of the University of Rome Tor Vergata.}}
\author{Xiaohong Lan \\
Institute of Mathematics, Chinese Academy of Sciences\\
and Department of Mathematics, University of Rome Tor Vergata\\
Email: lan@mat.uniroma2.it \and Domenico Marinucci \\
Department of Mathematics, University of Rome Tor Vergata\\
Email: marinucc@mat.uniroma2.it}
\maketitle

\begin{abstract}
We consider the correlation structure of the random coefficients for a class
of wavelet systems on the sphere (labelled Mexican needlets) which was
recently introduced in the literature by \cite{gm}. We provide necessary and
sufficient conditions for these coefficients to be asymptotic uncorrelated
in the real and in the frequency domain. Here, the asymptotic theory is
developed in the high frequency sense. Statistical applications are also
discussed, in particular with reference to the analysis of Cosmological data.

\begin{itemize}
\item \textbf{Keywords and Phrases: }Spherical Random Fields, Wavelets,
Mexican Needlets, High Frequency Asymptotics, Cosmic Microwave Background
Radiation

\item \textbf{AMS\ 2000 Classification: }Primary 60G60; Secondary 62M40,
42C40, 42C10
\end{itemize}
\end{abstract}

\section{Introduction}

There is currently a rapidly growing literature on the construction of
wavelets systems on the sphere, see for instance \cite{freeden}, \cite{jfaa1}%
, \cite{dahlke}, \cite{poisson}, \cite{rosca}, \cite{jfaa3} and the
references therein. Some of these attempts have been explicitly motivated by
extremely interesting applications, for instance in the framework of
Astronomy and/or Cosmology; concerning the latter, special emphasis has been
devoted to wavelet techniques for the statistical study of the Cosmic
Microwave Background (CMB) radiation (\cite{jfaa2}). CMB can be viewed as
providing observations on the Universe in the immediate adjacency of the Big
Bang, and as such it has been the object of immense theoretical and applied
interest over the last decade \cite{dode2004}.

Among spherical wavelets, particular attention has been devoted to so-called
needlets, which were introduced into the Functional Analysis literature by %
\cite{npw1,npw2}; their statistical properties were first considered by \cite%
{bkmpb}, \cite{bkmp}, \cite{fay08} and \cite{ejslan}. In particular, it has
been shown in \cite{bkmpb} that \ random needlet coefficients enjoy a
capital uncorrelation property: namely, for any fixed angular distance,
random needlets coefficients are asymptotically uncorrelated as the
frequency parameter grows larger and larger. The meaning of this
uncorrelation property must be carefully understood, given the specific
setting of statistical inference in Cosmology. Indeed, CMB can be viewed as
a single realization of an isotropic random field on a sphere of a finite
radius (\cite{dode2004}). The asymptotic theory is then entertained in the
high frequency sense, i.e. it is considered that observations at higher and
higher frequencies (smaller and smaller scales) become available with the
growing sophistication of CMB satellite experiments. Of course,
uncorrelation entails independence in the Gaussian case: as a consequence,
from the above-mentioned property it follows that an increasing array of
asymptotically $i.i.d.$ coefficients can be derived out of a single
realization of a spherical random field, making thus possible the
introduction of a variety of statistical procedures for testing
non-Gaussianity, estimating the angular power spectrum, testing for
asymmetries, implementing bootstrap techniques, testing for
cross-correlation among CMB\ and Large Scale Structure data, and many
others, see for instance \cite{bkmpb}, \cite{bkmp}, \cite{guilloux}, \cite%
{ejslan}, \cite{fay08}, \cite{pietrobon}, \cite{mpb08}, \cite{dela08}, \cite%
{betoule08}, \cite{lan2008}, \cite{rudjord}. We note that the relevance of
high frequency asymptotics is not specific to Cosmology (cf. e.g. financial
data).

Given such a widespread array of techniques which are made feasible by means
of the uncorrelation property, it is natural to investigate to what extent
this property should be considered unique for the construction in \cite%
{npw1,npw2}, or else whether it is actually shared by other proposals. In
particular, we shall focus here on the approach which has been very recently
advocated by \cite{gm}, see also \cite{freeden} for a related setting. This
approach (which we shall discuss in Section 2) can be labelled Mexican
needlets, for reasons to be made clear later. Its analysis is made
particularly interesting by the fact that, as we shall discuss below, the
Mexican needlets can be considered asymptotically equivalent to the
Spherical Mexican Hat Wavelet (SMHW), which is currently the most popular
wavelet procedure in the Cosmological literature (see again \cite{jfaa2}).
As such, the investigation of their properties will fill a theoretical gap
which is certainly of interest for CMB data analysis.

Our aim in this paper is then to investigate the correlation properties of
the Mexican needlets coefficients. The stochastic properties of wavelets
have been very extensively studied in the mathematical statistics
literature, starting from the seminal papers \cite{doukhan, doukhan2}. We
must stress, however, that our framework here is very different: indeed, as
explained earlier we shall be concerned with circumstances where
observations are made at higher and higher frequencies on a single
realization of a spherical random field. As such, no form of mixing or
related properties can be assumed on the data and the proofs will rely more
directly on harmonic methods, rather than on standard probabilistic
arguments.

We shall provide both a positive and a negative result: namely, we will
provide necessary and sufficient conditions for the Mexican needlets
coefficients to be uncorrelated, depending on the behaviour of the angular
power spectrum of the underlying (mean square continuous and isotropic)
random fields. In particular, on the contrary of what happens for the
needlets in \cite{npw1,npw2}, we shall show that there is indeed correlation
of the random coefficients when the angular power spectrum is decaying
faster than a certain limit. However, higher order versions (already
considered in \cite{gm}) of the Mexican needlets can indeed provide
uncorrelated coefficients, depending on a parameter which is related to the
decay of the angular power spectrum. In some sense, a heuristic rationale
under these results can be explained as follows: the correlation among
coefficients is introduced basically by the presence in each of these terms
of random elements which are fixed (with respect to growing frequencies) in
a given realization of the random field, because they depend only on very
large scale behaviour (this is known in the Physical literature as a Cosmic
Variance effect). Because of the compact support in frequency in the
needlets as developed by \cite{npw1, npw2}, these low-frequency components
are always dropped and uncorrelation is ensured. On the other hand, the same
components can be dominant for Mexican needlets, in which cases it becomes
necessary to introduce suitably modified versions which are better localized
in the frequency domain (i.e., they allow less weight on very low frequency
components).

As well-known, there is usually a trade-off between localization properties
in the frequency and real domains, as a consequence of the Uncertainty
Principle (``It is impossible for a non-zero function and its Fourier
transform to be simultaneosuly very small'', see for instance \cite{havin}).
In view of this, an interesting consequence of our results can be loosely
suggested as follows: the better the localization in real domain, the worst
the correlation properties. This is clearly a paradox, and we do not try to
formulate it more rigorously from the mathematical point of view - some
numerical evidence will be collected in an ongoing, more applied work.
However, we hope that the previous discussion will help to shed some light
within the class of needlets; in particular, it should clarify that the
uncorrelation property of wavelets coefficients does not follow at all by
their localization properties in real domain. Indeed, given the fixed-domain
asymptotics we are considering, perfect localization in real space does not
ensure any form of uncorrelation (all random values at different locations
on the sphere have in general a non-zero correlation).

The plan of the paper is as follows: in Section 2 we shall \ review some
basic results on isotropic random fields on the sphere and the (Mexican and
standard) needlets constructions. In Section 3 and 4 we establish our main
results, providing necessary and sufficient conditions for the uncorrelation
properties to hold; in Section 5 we review some statistical applications.

\section{Isotropic Random Fields and Spherical Needlets}

\subsection{Spherical Random Fields}

In this paper, we shall always be concerned with zero-mean, finite variance
and isotropic random fields on the sphere, for which the following spectral
representation holds, in the mean square sense:

\begin{equation}
T(x)=\sum_{lm}a_{lm}Y_{lm}(x)\text{ , }x\in S^{2},  \label{specrap}
\end{equation}%
where $\left\{ a_{lm}\right\} _{l,m},$ $m=-l,...,l$ is a triangular array of
zero-mean, orthogonal, complex-valued (for $m\neq 0$) random variables with
variance $\mathbb{E}|a_{lm}|^{2}=C_{l},$ the angular power spectrum of the
random field. The functions $\left\{ Y_{lm}(x)\right\} $ are the so-called
spherical harmonics, i.e. the eigenvectors of the Laplacian operator on the
sphere \cite{faraut}, \cite{vmk}, \cite{stein} 
\begin{eqnarray*}
\Delta _{S^{2}}Y_{lm}(\vartheta ,\varphi ) &=&\left[ \frac{1}{\sin \vartheta 
}\frac{\partial }{\partial \vartheta }\left\{ \sin \vartheta \frac{\partial 
}{\partial \vartheta }\right\} +\frac{1}{\sin ^{2}\vartheta }\frac{\partial
^{2}}{\partial \varphi ^{2}}\right] Y_{lm}(\vartheta ,\varphi ) \\
&=&-l(l+1)Y_{lm}(\vartheta ,\varphi )
\end{eqnarray*}%
where we have moved to spherical coordinates $x=(\vartheta ,\varphi )$, $%
0\leq \vartheta \leq \pi $ and $0\leq \varphi <2\pi $ . It is a well-known
result that the spherical harmonics provide a complete orthonormal systems
for $L^{2}(S^{2}).$ There are many routes for establishing (\ref{specrap}),
usually by means of Karhunen-Lo\'{e}ve arguments, the Spectral
Representation Theorem, or the Stochastic Peter-Weyl theorem, see for
instance \cite{adler}. The spherical harmonic coefficients $a_{lm}$ can be
recovered by means of the Fourier inversion formula%
\begin{equation}
a_{lm}=\int_{S^{2}}T(x)\overline{Y_{lm}(x)}dx\text{ .}  \label{invfor}
\end{equation}%
If the random field is mean-square continuous, the angular power spectrum $%
\left\{ C_{l}\right\} $ must satisfy the summability condition 
\begin{equation*}
\sum_{l}(2l+1)C_{l}<\infty \text{ .}
\end{equation*}%
We shall introduce a slightly stronger condition, as follows (see \cite%
{bkmpb, bkmp,fay08,ejslan}).

\begin{condition}
\label{cdA} For all $B>1,$ there exist $\alpha >2,$ and $\left\{
g_{j}(.)\right\} _{j=1,2,...}$ a sequence of functions such that%
\begin{equation}
C_{l}=l^{-\alpha }g_{j}(\frac{l}{B^{j}})>0\text{ , for }B^{j-1}<l<B^{j+1}%
\text{ , }j=1,2,...  \label{Specon1}
\end{equation}%
where 
\begin{equation*}
c_{0}^{-1}\leq g_{j}\leq c_{0}\text{ for all }j\in \mathbb{N}\text{ , and }%
\sup_{j}\sup_{B^{-1}\leq u\leq B}|\frac{d^{r}}{du^{r}}g_{j}(u)|\leq c_{r}%
\text{ ,}
\end{equation*}%
\begin{equation*}
\text{some }c_{0},c_{1},...c_{M}>0\text{ },\text{ }M\in \mathbb{N}\text{ .}
\end{equation*}
\end{condition}

In practice, random fields such as CMB are not fully observed, i.e. there
are some missing observations in some regions of $S^{2};$ (\ref{invfor}) is
thus unfeasible in its exact form, and this motivates the introduction of
spherical wavelets such as needlets.

\subsection{NPW Needlets}

The construction of the standard needlet system is detailed in \cite%
{npw1,npw2}; we can label this system \emph{NPW needlets} and we sketch here
a few details for completeness. Let $\phi $ be a $C^{\infty }$ function
supported in $|\xi |\leq 1$, such that $0\leq \phi (\xi )\leq 1$ and $\phi
(\xi )=1$ if $|\xi |\leq 1/B$, $B>1$. Define 
\begin{equation}
b^{2}(\xi )=\phi (\frac{\xi }{B})-\phi (\xi )\geq 0\text{ so that }\forall
|\xi |\geq 1\text{ },\text{ }\sum_{j=0}^{\infty }b^{2}(\frac{\xi }{B^{j}})=1%
\text{ .}  \label{bdef}
\end{equation}%
It is immediate to verify that $b(\xi )\neq 0$ only if $\frac{1}{B}\leq |\xi
|\leq B$. The needlets frame $\left\{ \varphi _{jk}(x)\right\} $ is then
constructed as 
\begin{equation}
\varphi _{jk}(x):=\sqrt{\lambda _{jk}}\sum_{l}b(\frac{l}{B^{j}}%
)\sum_{m=-l}^{l}Y_{lm}(\xi _{jk})\overline{Y_{lm}}(x)\text{ .}  \label{3}
\end{equation}%
Here, $\left\{ \lambda _{jk}\right\} $ is a set of \emph{cubature weights}
corresponding to \emph{the cubature points} $\left\{ \xi _{jk}\right\} ;$
they are such to ensure that, for all polynomials $Q_{l}(x)$ of degree
smaller than $B^{j+1}$ 
\begin{equation*}
\sum_{k}Q_{l}(\xi _{jk})\lambda _{jk}=\int_{S^{2}}Q_{l}(x)dx\text{ .}
\end{equation*}%
The main localization property of $\left\{ \varphi _{jk}(x)\right\} $ is
established in \cite{npw1}, where it is shown that for any $M\in \mathbb{N}$
there exists a constant $c_{M}>0$ s.t., for every $\xi \in \mathbb{S}^{2}$: 
\begin{equation*}
\left| \varphi _{jk}(\xi )\right| \leq \frac{c_{M}B^{j}}{(1+B^{j}\arccos
\langle \xi _{jk},\xi \rangle )^{M}}\text{ uniformly in }(j,k)\text{ }.
\end{equation*}%
More explicitly, needlets are almost exponentially localized around any
cubature point, which motivates their name. In the stochastic case, the
(random) spherical needlet coefficients are then defined as 
\begin{equation}
\beta _{jk}=\int_{\mathbb{S}^{2}}T(x)\varphi _{jk}(x)dx=\sqrt{\lambda _{jk}}%
\sum_{l}b(\frac{l}{B^{j}})\sum_{m=-l}^{l}a_{lm}Y_{lm}(\xi _{jk})\text{ }.
\label{bjk}
\end{equation}%
It is then immediate to derive the correlation coefficient 
\begin{eqnarray*}
Corr\left( \beta _{jk},\beta _{jk^{\prime }}\right)  &=&\frac{E\beta
_{jk}\beta _{jk^{\prime }}}{\sqrt{E\beta _{jk}^{2}E\beta _{jk^{\prime }}^{2}}%
} \\
&=&\frac{\sqrt{\lambda _{jk}\lambda _{jk^{\prime }}}\underset{l\geq 1}{\sum }%
b^{2}(\frac{l}{B^{j}})\frac{2l+1}{4\pi }C_{l}P_{l}\left( \left\langle \xi
_{jk},\xi _{jk^{\prime }}\right\rangle \right) }{\sqrt{\lambda _{j,k}\lambda
_{j,k^{\prime }}}\underset{l\geq 1}{\sum }b^{2}(\frac{l}{B^{j}})\frac{2l+1}{%
4\pi }C_{l}}\text{ .}
\end{eqnarray*}%
where $P_{l}$ is the ultraspherical (or Legendre) polynomial of order $\frac{%
1}{2}$ and degree $l$; the last step follows from the well-known identity %
\cite{vmk}

\begin{equation*}
L_{l}\left( \left\langle \xi ,\eta \right\rangle \right) :=\overset{l}{%
\underset{m=-l}{\sum }}Y_{lm}\left( \xi \right) \overline{Y_{lm}\left( \eta
\right) }=\frac{2l+1}{4\pi }P_{l}\left( \left\langle \xi ,\eta \right\rangle
\right) \text{ }.
\end{equation*}%
The capital stochastic property for random needlet coefficients is provided
by \cite{bkmpb}, where it is shown that under Condition \ref{cdA} the
following inequality holds%
\begin{equation}
\left| Corr\left( \beta _{jk},\beta _{jk^{\prime }}\right) \right| \leq 
\frac{C_{M}}{\left( 1+B^{j}d\left( \xi _{jk},\xi _{jk^{\prime }}\right)
\right) ^{M}},\text{ some }C_{M}>0\text{ , }  \label{corr1}
\end{equation}%
where $d\left( \xi _{jk},\xi _{jk^{\prime }}\right) =\arccos \left(
\left\langle \xi _{jk},\xi _{jk^{\prime }}\right\rangle \right) $ is the
standard distance on the sphere.

\subsection{Mexican Needlets}

The construction in \cite{gm} is in a sense similar to NPW needlets in \cite%
{npw1,npw2} (see also \cite{freeden}), insofar as a combination of Legendre
polynomials with a smooth function is proposed; the main difference is that
for NPW\ needlets the kernel is taken to be compactly supported, which
allows on one hand for an exact reconstruction function (needlets make up a
tight frame), at the same time granting exact localization in the
frequency-domain. It should be added, however, that the approach by \cite{gm}
enjoys some undeniable strong points: firstly, it covers general oriented
manifolds and not simply the sphere; moreover it yields Gaussian
localization properties in the real domain. A further nice benefit is that
it can be formulated in terms of an explicit recipe in real space, a feature
which is certainly valuable for practitioners. In particular, as we report
below in the high-frequency limit the Mexican needlets are asymptotically
close to the Spherical Mexican Hat Wavelets, which have been exploited in
several Cosmological papers but still lack a sound stochastic investigation.

More precisely, \cite{gm} propose to replace $b(l/B^{j})$ in (\ref{3}) by $%
f(l(l+1)/B^{2j}),$ where $f(.)$ is some Schwartz function vanishing at zero
(not necessarily of bounded support) and the sequence $\left\{
-l(l+1)\right\} _{l=1,2,...}$ represents the eigenvalues of the Laplacian
operator $\Delta _{S^{2}}.$ In particular, Mexican needlets can be obtained
by taking $f(s)=s\exp (-s),$ so to obtain%
\begin{equation*}
\psi _{jk;1}\left( x\right) =\sqrt{\lambda _{jk}}\underset{l\geq 1}{\sum }%
\frac{l\left( l+1\right) }{B^{2j}}e^{-\frac{l\left( l+1\right) }{B^{2j}}%
}L_{l}\left( \left\langle x,\xi _{jk}\right\rangle \right) \text{ .}
\end{equation*}%
The resulting functions make up a frame which is not tight, but very close
to, in a sense which is made rigorous in \cite{gm}. Exact cubature formulae
cannot hold (in particular, $\left\{ \lambda _{jk}\right\} $ are not exactly
cubature weights in this case), because polynomials of infinitely large
order are involved in the construction, but again this entails very minor
approximations in practical terms. More generally, it is possible to
consider higher order Mexican needlets by focussing on $f(s)=s^{p}\exp (-s),$
so to obtain%
\begin{equation*}
\psi _{jk;p}\left( x\right) :=\sqrt{\lambda _{jk}}\underset{l\geq 1}{\sum }(%
\frac{l(l+1)}{B^{2j}})^{p}e^{-l(l+1)/B^{2j}}L_{l}\left( \left\langle x,\xi
_{jk}\right\rangle \right) \text{ .}
\end{equation*}%
The random spherical Mexican needlet coefficients are immediately seen to be
given by%
\begin{eqnarray*}
\beta _{jk;p} &=&\int_{S^{2}}T\left( x\right) \psi _{jk}\left( x\right)
dx=\int_{S^{2}}\underset{l\geq 0}{\sum }\overset{l}{\underset{m=-l}{\sum }}%
a_{lm}Y_{lm}\left( x\right) \psi _{jk;p}\left( x\right) dx \\
&=&\sqrt{\lambda _{jk}}\underset{l\geq 1}{\sum }(\frac{l(l+1)}{B^{2j}}%
)^{p}e^{-l(l+1)/B^{2j}}\overset{l}{\underset{m=-l}{\sum }}a_{lm}Y_{lm}\left(
\xi _{jk}\right) \text{ },
\end{eqnarray*}%
whence their covariance is%
\begin{equation*}
E\beta _{jk;p}\beta _{jk^{\prime };p}=\sqrt{\lambda _{jk}\lambda
_{jk^{\prime }}}\underset{l\geq 1}{\sum }(\frac{l(l+1)}{B^{2j}}%
)^{p}e^{-2l(l+1)/B^{2j}}\frac{2l+1}{4\pi }C_{l}P_{l}\left( \left\langle \xi
_{jk},\xi _{jk^{\prime }}\right\rangle \right) .
\end{equation*}%
Throughout this paper, we shall only consider weight functions of the form $%
f(s)=s^{p}\exp (-s).$ It is certainly possible to consider more general
constructions; however this specific shape lends itself to very neat
results, allowing us to produce both upper and lower bounds for the
coefficients' correlation. Also, it makes possible a clear interpretation of
the final results, i.e. the effect of varying $p$ on the structure of
dependence is immediately understood; this is, we believe, a valuable asset
for practitioners. In the sequel, we shall drop the subscript $p$, whenever
possible without risk of confusion.

\ \newline

\begin{remark}
As mentioned earlier, it is suggested from results in (\cite{gm}) that
Mexican needlets provide asymptotically a very good approximation to the
widely popular Spherical Mexican Hat Wavelets (SMHW), which have been used
in many physical papers; the asymptotic analysis of the stochastic
properties of SMHW coefficients is still completely open for research. The
discretized form of the SMHW\ can be written as 
\begin{equation*}
\Psi _{jk}(\theta ;B^{-j})=\frac{1}{(2\pi )^{\frac{1}{2}}\sqrt{2}%
B^{-j}(1+B^{-2j}+B^{-4j})^{\frac{1}{2}}}[1+(\frac{y}{2})^{2}]^{2}[2-\frac{%
y^{2}}{2t^{2}}]e^{-y^{2}/4B^{-j2}},
\end{equation*}%
where the coordinates $y=2\tan \frac{\theta }{2}$ follows from the
stereographic projection on the tangent plane in each point of the sphere;
here we take $\theta =\theta _{jk}\left( x\right) :=d(x,\xi _{jk})$. Now
write%
\begin{equation*}
\psi _{jk;p}\left( \theta _{jk}\left( x\right) \right) =\psi _{jk;p}\left(
\theta \right) \text{ ;}
\end{equation*}%
by following the arguments in \cite{gm} and developing their bounds further,
it can be argued that%
\begin{equation}
\left| \Psi _{jk}(\theta ;B^{-j})-K_{jk}\psi _{jk;p}\left( \theta \right)
\right| =B^{-j}O\left( \min \left\{ \theta ^{4}B^{4j},1\right\} \right) 
\text{ ,}  \label{lan}
\end{equation}%
for some suitable normalization constant $K_{jk}>0.$ Equation (\ref{lan})
suggests that our results below can be used as a guidance for the asymptotic
theory of random SMHW\ coefficients. The validity of this approximation over
relevant cosmological models and its implications for statistical procedures
of CMB\ data analysis are currently being investigated.
\end{remark}

\section{Stochastic properties of Mexican needlet coefficients, I: upper
bounds}

As mentioned in the Introduction, having established ($\ref{corr1}$) opened
the way to several developments for the statistical analysis of spherical
random fields. It is therefore a very important question to establish under
what circumstances these results can be extended to other constructions,
such as Mexican needlets. In this and the following Section, we provide a
full characterization with positive and negative results. We start by
writing the expression for the correlation coefficients, which is given by%
\begin{equation*}
Corr\left( \beta _{j_{1}k_{1}},\beta _{j_{2}k_{2}}\right)
\end{equation*}%
\begin{equation*}
=\frac{\underset{l\geq 1}{\sum }(\frac{l(l+1)}{B^{j_{1}}})^{p}(\frac{l(l+1)}{%
B^{j_{2}}})^{p}e^{-l(l+1)(B^{-2j_{1}}+B^{-2j_{2}})}(2l+1)C_{l}P_{l}\left(
\left\langle \xi _{j_{1}k_{1}},\xi _{j_{2}k_{2}}\right\rangle \right) }{%
\left\{ \underset{l\geq 1}{\sum }(\frac{l(l+1)}{B^{j_{1}}}%
)^{4p}e^{-2l(l+1)/B^{2j_{1}}}(2l+1)C_{l}\right\} ^{1/2}\left\{ \underset{%
l\geq 1}{\sum }(\frac{l(l+1)}{B^{j_{2}}}%
)^{4p}e^{-2l(l+1)/B^{2j_{2}}}(2l+1)C_{l}\right\} ^{1/2}}\text{ .}
\end{equation*}%
We now provide upper bounds on the correlation of random coefficients, as
follows\footnote{%
While finishing this paper, we learned by personal communication that
working independently and at the same time as us, A.Mayeli has obtained a
result similar to Theorem \ref{uno} for the case $j_{1}=j_{2}$, see \cite%
{mayeli}. The statements and the assumptions in the two approaches are not
equivalent and the methods of proofs are entirely different; we believe both
are of independent interest.}.

\begin{theorem}
$\label{uno}$Assume Condition \ref{cdA} holds with $\alpha <4p+2$ and $M\geq
4p+2-\alpha ;$ then there exist some constant $C_{M}>0$ 
\begin{equation}
\left\vert Corr\left( \beta _{j_{1}k_{1};p},\beta _{j_{2}k_{2};p}\right)
\right\vert \leq \frac{C_{M}}{\left( 1+B^{(j_{1}+j_{2})/2-\log
_{B}(j_{1}+j_{2})/2}d\left( \xi _{jk},\xi _{jk^{\prime }}\right) \right)
^{\left( 4p+2-\alpha \right) }}.  \label{3.1}
\end{equation}
\end{theorem}

\begin{proof}
We prove (\ref{3.1}) following some ideas in \cite{npw1}. For notational
simplicity, we focus first on the case where $j_{1}=j_{2};$ we have%
\begin{equation}
Corr\left( \beta _{jk;p},\beta _{jk^{\prime };p}\right) =\frac{\underset{%
l\geq 1}{\sum }(\frac{l\left( l+1\right) }{B^{2j}})^{2p}e^{-2l(l+1)/B^{2j}}%
\frac{2l+1}{4\pi }C_{l}P_{l}\left( \left\langle \xi _{jk},\xi _{jk^{\prime
}}\right\rangle \right) }{\underset{l\geq 1}{\sum }(\frac{l\left( l+1\right) 
}{B^{2j}})^{2p}e^{-2l^{2}/B^{2j}}\frac{2l+1}{4\pi }C_{l}}.  \label{cor2}
\end{equation}%
Now replace $C_{l}=l^{-\alpha }g_{j}\left( \frac{l}{B^{j}}\right) $ in the
denominator of the above representation, for which we get%
\begin{equation*}
c_{0}^{-1}\frac{B^{-\left( 4p+2-\alpha \right) j}}{4p+2-\alpha }\leq
\sum_{l\geq 1}(\frac{l\left( l+1\right) }{B^{2j}})^{2p}e^{-2l(l+1)/B^{2j}}%
\frac{2l+1}{4\pi }C_{l}\leq c_{0}\frac{B^{-\left( 4p+2-\alpha \right) j}}{%
4p+2-\alpha }\text{ }.
\end{equation*}%
Denoting $\theta =\arccos \left\langle \xi _{j,k},\xi _{j,k^{\prime
}}\right\rangle $, the numerator can be written as%
\begin{equation}
\underset{l\geq 1}{\sum }(\frac{l\left( l+1\right) }{B^{2j}}%
)^{2p}e^{-2l(l+1)/B^{2j}}\frac{2l+1}{4\pi }C_{l}\frac{1}{\pi }\int_{\theta
}^{\pi }\frac{\sin \left( l+\frac{1}{2}\right) \varphi }{\left( \cos \theta
-\cos \varphi \right) ^{1/2}}d\varphi  \label{num}
\end{equation}%
where we have used the Dirichlet-Mehler integral representation for the
Legendre polynomials \cite{gra}.

The following steps and notations are very close to \cite{npw1}. We write%
\begin{eqnarray}
C_{B,g_{j}} &=&\underset{l\geq 1}{\sum }(\frac{l\left( l+1\right) }{B^{2j}}%
)^{2p}e^{-2l(l+1)/B^{2j}}\frac{2l+1}{4\pi }l^{-\alpha }g_{j}\left( \frac{l}{%
B^{j}}\right) \sin \left( l+\frac{1}{2}\right) \varphi  \label{cbj} \\
&:&=\frac{1}{2i}\underset{l\geq 1}{\sum }\left( h_{j+}(l)-h_{j-}\left(
l\right) \right)  \notag
\end{eqnarray}%
where%
\begin{equation*}
h_{j\pm }(u)=(\frac{u\left( u+1\right) }{B^{2j}})^{2p}\frac{2u+1}{4\pi }%
u^{-\alpha }g_{j}\left( \frac{u}{B^{j}}\right) e^{-2u(u+1)/B^{2j}\pm i(u+%
\frac{1}{2})\varphi }.
\end{equation*}%
By Poisson summation formula, we have%
\begin{equation*}
\underset{l\geq 1}{\sum }h_{j\pm }(l)=\frac{1}{2}\underset{l\in 
\mathbb{Z}
}{\sum }h_{j\pm }(l)=\frac{1}{2}\underset{\mu \in 
\mathbb{Z}
}{\sum }\widehat{h}_{j\pm }\left( 2\pi \mu \right) \text{ .}
\end{equation*}%
Denote%
\begin{equation}
G_{\alpha ,j}\left( t\right) :=t^{2p-\alpha }g_{j}\left( t\right)
e^{-2t(t+B^{-j})}I_{\left( B^{-j},\infty \right) }\text{ .}  \label{gbj}
\end{equation}%
Let us now recall the following standard property of Fourier transforms:%
\begin{equation*}
\left( i\omega \right) ^{k}\frac{d^{m}}{d\omega ^{m}}\widehat{f}\left(
\omega \right) =(\frac{d^{k}}{dx^{k}}\left\{ x^{m}f\left( x\right) \right\} 
\widehat{)(}\omega )\text{ }.
\end{equation*}%
Some simple computations yield%
\begin{eqnarray*}
\widehat{h}_{\pm }\left( 2\pi \mu \right) &=&\frac{B^{\left( 1-\alpha
\right) j}}{4\pi }\left\{ \sum_{m=1}^{2p}[2\binom{2p}{m-1}+\binom{2p}{m}%
]B^{-\left( 2p+1-m\right) j}\frac{d^{m}}{d\omega ^{m}}+B^{-\left(
2p+1\right) j}+2\frac{d^{2p+1}}{d\omega ^{2p+1}}\right\} \\
&&\times \int_{-\infty }^{\infty }G_{\alpha ,j}\left( t/B^{j}\right) e^{\pm
i(t+\frac{1}{2})\varphi -it\omega }dt\left| _{\omega =2\pi \mu }\right.
\end{eqnarray*}%
\begin{eqnarray*}
&=&\frac{B^{\left( 2-\alpha \right) j}e^{\pm i\varphi }}{4\pi }\left\{
\sum_{m=1}^{2p}[2\binom{2p}{m-1}+\binom{2p}{m}]B^{-\left( 2p+1-m\right) j}%
\frac{d^{m}}{d\omega ^{m}}\right\} \widehat{G}_{\alpha ,j}\left( \omega
\right) \left| _{\omega =B^{j}\left( 2\pi \mu \mp \varphi \right) }\right. \\
&&+\frac{B^{\left( 2-\alpha \right) j}e^{\pm i\varphi }}{4\pi }\left\{
B^{-\left( 2p+1\right) j}+2\frac{d^{2p+1}}{d\omega ^{2p+1}}\right\} \widehat{%
G}_{\alpha ,j}\left( \omega \right) \left| _{\omega =B^{j}\left( 2\pi \mu
\mp \varphi \right) }\right. ,
\end{eqnarray*}%
where%
\begin{equation*}
\widehat{G}_{\alpha ,j}\left( \omega \right) =\int_{%
\mathbb{R}
}G_{\alpha ,j}\left( t\right) e^{-it\omega }dt\text{ }.
\end{equation*}%
For all positive integers $k\leq M,$ we can obtain%
\begin{equation}
\left| \int_{B^{-j}}^{\infty }\frac{d^{k}}{dt^{k}}\left\{ t^{m}G_{\alpha
,j}\left( t\right) \right\} dt\right| \leq \left\{ 
\begin{array}{c}
\frac{k\Gamma \left( m+2p-\alpha \right) C_{g}}{L(p,m,\alpha ,k)}B^{-j\left(
2p+1+m-\alpha -k\right) },\text{ for }2p+1-\alpha +m\neq k \\ 
k\Gamma \left( m+2p-\alpha \right) C_{g}(j\log B)\text{ },\text{ for }%
2p+1-\alpha +m=k%
\end{array}%
\right.  \label{intlog}
\end{equation}%
where $C_{g}=\max \left\{ c_{0},...,c_{M}\right\} $ and 
\begin{equation*}
L(p,m,\alpha ,k):=\left\{ 
\begin{array}{c}
\Gamma \left( 2p+1+m-\alpha -k\right) \text{ when }\left( 2p+1+m-\alpha
-k\right) >0\text{ ,} \\ 
\left( \Gamma \left( k-2p-1-m+\alpha \right) \right) ^{-1}\text{ when }%
\left( 2p+1+m-\alpha -k\right) <0\text{ .}%
\end{array}%
\right.
\end{equation*}%
It should be noticed that our argument here differs from the one in \cite%
{npw1}, because we cannot assume the integrand function on the left-hand
side to be in $L^{1}$ for all $k\leq M.$ Let us now focus on the case where $%
k=4p+2-\alpha ,$ with $m=2p+1;$ we obtain%
\begin{eqnarray*}
&&\left| \frac{d^{2p+1}}{d\omega ^{2p+1}}\widehat{G}_{\alpha ,j}\left(
\omega \right) \right| \left| B^{j}\left( 2\pi \mu -\varphi \right) \right|
^{4p+2-\alpha } \\
&\leq &\left| \int_{B^{-j}}^{\infty }\frac{d^{k}}{dt^{k}}\left\{
t^{2p+1}G_{j}\left( t\right) \right\} dt\right| \leq \left( 4p+2-\alpha
\right) \Gamma \left( 4p+1-\alpha \right) C_{g}(j\log B);
\end{eqnarray*}%
therefore%
\begin{equation*}
\left| \frac{d^{2p+1}}{d\omega ^{2p+1}}\widehat{G}_{\alpha ,j}\left( \omega
\right) \right| \leq \frac{C_{2p+1,\alpha ,g}B^{-j\left( 4p+2-\alpha \right)
+\log _{B}j}}{\left( 2\pi \mu -\varphi \right) ^{4p+2-\alpha }},
\end{equation*}%
where%
\begin{equation*}
C_{2p+1,\alpha ,g}=\left( 4p+2-\alpha \right) \Gamma \left( 4p+1-\alpha
\right) C_{g}\log B\text{ }.
\end{equation*}%
The modifications needed for other cases are obvious and we obtain%
\begin{equation*}
B^{-\left( 2p+1-m\right) j}\left| \frac{d^{m}}{d\omega ^{m}}\widehat{G}%
_{\alpha ,j}\left( \omega \right) \right| \leq \frac{C_{m,\alpha
,g}B^{-j\left( 4p+2-\alpha \right) }}{\left( 2\pi \mu -\varphi \right)
^{4p+2-\alpha }},
\end{equation*}%
where%
\begin{equation*}
C_{m,\alpha ,g}=\frac{\left( 4p+2-\alpha \right) \Gamma \left( m+2p-\alpha
\right) C_{g}}{L(p,m,\alpha ,k)}\text{ }.
\end{equation*}%
Now let $C_{\alpha ,g}=\max \left\{ C_{m,\alpha ,g},\text{ }%
m=0,...,2p+1\right\} ;$ we have%
\begin{eqnarray}
&&\left| \widehat{h}_{\pm }\left( 2\pi \mu \right) \right|  \notag \\
&\leq &\frac{B^{\left( 2-\alpha \right) j}}{4\pi }C_{\alpha ,g}\left\{
\sum_{m=0}^{2p}\binom{2p}{m}\frac{B^{-j\left( 4p+2-\alpha \right) }}{\left(
2\pi \mu -\varphi \right) ^{4p+2-\alpha }}+\frac{B^{-j\left( 4p+2-\alpha
\right) +\log _{B}j}}{\left( 2\pi \mu -\varphi \right) ^{4p+2-\alpha }}%
\right\}  \notag \\
&\leq &\frac{2^{2p}C_{\alpha ,g}B^{-4pj+\log _{B}j}}{\left( 2\pi \mu \mp
\varphi \right) ^{4p+2-\alpha }},\text{ }\mu =1,2,...  \label{newproof}
\end{eqnarray}%
Therefore 
\begin{eqnarray*}
\left| C_{B,g_{j}}\right| &\leq &2^{2p}C_{\alpha ,g}\left( \frac{1}{2\varphi
^{4p+2-\alpha }}+\underset{\mu \in N}{\sum }\frac{1}{\left| 2\pi \mu \pm
\varphi \right| ^{4p+2-\alpha }}\right) B^{-4pj+\log _{B}j} \\
&\leq &2^{2p}\left( \frac{1}{2\varphi ^{4p+2-\alpha }}+\left| 4p+1-\alpha
\right| \pi ^{\alpha -4p-1}\right) C_{\alpha ,g}B^{-4pj+\log _{B}j}.
\end{eqnarray*}%
Hence, for the numerator of the correlation we have the bound%
\begin{equation*}
(\ref{num})\leq CB^{-4pj+\log _{B}j}\int_{\theta }^{\pi }\frac{\left( \frac{1%
}{2\varphi ^{4p+2-\alpha }}+\left| 4p+1-\alpha \right| \pi ^{\alpha
-4p-1}\right) }{\left( \cos \theta -\cos \varphi \right) ^{1/2}}d\varphi 
\text{ }.
\end{equation*}%
As in \cite{npw1}, when $0\leq \theta \leq \pi /2,$ we can get the following
inequality%
\begin{equation*}
(\ref{num})\leq C_{\alpha ,k,g}B^{-4pj+\log _{B}j}\int_{\theta }^{\pi }\frac{%
\pi ^{\alpha -4p-1}}{0.27\varphi ^{4p+2-\alpha }\left( \varphi ^{2}-\theta
^{2}\right) ^{1/2}}d\varphi \leq C_{1}B^{-4pj+\log _{B}j}\theta ^{\alpha
-4p-2}\text{ }.
\end{equation*}%
If $\pi /2\leq \theta \leq \pi ,$ letting $\widetilde{\theta }=\pi -\theta ,%
\widetilde{\varphi }=\pi -\varphi ,$ we can obtain the same bound. Going
back to (\ref{cor2}), we obtain%
\begin{eqnarray*}
Corr\left( \beta _{j,k},\beta _{j,k^{\prime }}\right) &\leq &C\frac{\theta
^{\alpha -4p-2}B^{-4pj+\log _{B}j}}{B^{\left( 2-\alpha \right) j}} \\
&\leq &C\theta ^{\alpha -4p-2}B^{-\left( j-\log _{B}j\right) \left(
4p+2-\alpha \right) }\rightarrow 0,\text{ }as\text{ }j\rightarrow \infty 
\text{ ,}
\end{eqnarray*}

We thus get inequality (\ref{3.1}) for $j_{1}=j_{2}$.

Now let us consider $j_{1}\neq j_{2}.$ As the proof is very similar to the
arguments above, we omit many details.\ In the sequel, for any two sequences 
$a_{l},b_{l},$ we write $a_{l}\approx b_{l}$ if and only if $a_{l}=O(b_{l})$
and $b_{l}=O(a_{l}).$

First, we consider the variance of the random coefficients, which can be
represented by:%
\begin{equation*}
\sum (\frac{l(l+1)}{B^{j_{1}}})^{4p}e^{-2l(l+1)/B^{2j_{1}}}(2l+1)C_{l}%
\approx B^{\left( 2-\alpha \right) j_{1}}\left\{ \left(
\int_{B^{-j_{1}}}^{1}+\int_{1}^{\infty }\right) t^{4p+1-\alpha
}e^{-2t(t+B^{-j_{1}})}g_{j_{1}}\left( t\right) dt\right\}
\end{equation*}%
\begin{equation*}
=B^{\left( 2-\alpha \right) j_{1}}\left\{ \frac{c_{0}}{4p+2-\alpha }%
B^{-j_{1}\left( 4p+2-\alpha \right) }+O\left( 1\right) \right\} \text{ };
\end{equation*}%
therefore%
\begin{eqnarray*}
&&\left\{ \underset{l\geq 1}{\sum }(\frac{l(l+1)}{B^{j_{1}}}%
)^{4p}e^{-2l(l+1)/B^{2j_{1}}}(2l+1)C_{l}\right\} ^{1/2}\left\{ \underset{%
l\geq 1}{\sum }(\frac{l(l+1)}{B^{j_{2}}}%
)^{4p}e^{-2l(l+1)/B^{2j_{2}}}(2l+1)C_{l}\right\} ^{1/2} \\
&=&B^{\left( 1-\alpha /2\right) \left( j_{1}+j_{2}\right) }\left\{ \frac{%
c_{0}}{4p+2-\alpha }B^{-j_{1}\left( 4p+2-\alpha \right) }+O\left( 1\right)
\right\} ^{1/2}\left\{ \frac{c_{0}}{4p+2-\alpha }B^{j_{2}\left( 4p+2-\alpha
\right) }+O\left( 1\right) \right\} ^{1/2} \\
&=&O\left( 1\right) B^{\left( 1-\alpha /2\right) \left( j_{1}+j_{2}\right) }%
\text{ .}
\end{eqnarray*}%
Without loss of generality, we can always assume $j_{1}<j_{2}.$ We can
implement the same argument as before, provided we replace $C_{B,g_{j}}$ in (%
\ref{cbj}) by 
\begin{eqnarray*}
C_{B,g,j_{1},j_{2}} &=&\underset{l\geq 1}{\sum }(\frac{l\left( l+1\right) }{%
B^{j_{1}+j_{2}}})^{2p}e^{-l(l+1)(B^{-2j_{1}}+B^{-2j_{2}})}\frac{2l+1}{4\pi }%
l^{-\alpha }g_{j_{1}}\left( \frac{l}{B^{j_{1}}}\right) \sin \left( l+\frac{1%
}{2}\right) \varphi \\
&=&:\frac{1}{2i}\underset{l\geq 1}{\sum }\left(
h_{j_{1}j_{2}+}(l)-h_{j_{1}j_{2}-}\left( l\right) \right)
\end{eqnarray*}%
where%
\begin{equation*}
h_{j_{1}j_{2}\pm }(u)=(\frac{u\left( u+1\right) }{B^{j_{1}+j_{2}}})^{2p}%
\frac{2u+1}{4\pi }u^{-\alpha }g_{j_{1}}\left( \frac{u}{B^{j_{1}}}\right)
e^{-u(u+1)(B^{-2j_{1}}+B^{-2j_{2}})\pm i(u+\frac{1}{2})\varphi }.
\end{equation*}%
Again, by Poisson summation formula, we have%
\begin{equation*}
\underset{l\geq 1}{\sum }h_{j_{1}j_{2}\pm }(l)=\frac{1}{2}\underset{l\in 
\mathbb{Z}
}{\sum }h_{j_{1}j_{2}\pm }(l)=\frac{1}{2}\underset{\mu \in 
\mathbb{Z}
}{\sum }\widehat{h}_{j_{1}j_{2}\pm }\left( 2\pi \mu \right) \text{ .}
\end{equation*}%
Denote%
\begin{equation*}
G_{\alpha ,j_{1}j_{2}}\left( t\right) :=t^{2p-\alpha }g_{j_{1}}\left(
t\right) e^{-t\left( t+B^{-j_{1}}\right) \left( 1+B^{2\left(
j_{1}-j_{2}\right) }\right) }I_{\left( B^{-j_{2}},\infty \right) }\text{ ;}
\end{equation*}%
by the same argument and notation as in (\ref{newproof}), we have%
\begin{eqnarray*}
&&\left| \widehat{h}_{j_{1}j_{2}\pm }\left( 2\pi \mu \right) \right| \\
&=&\!\frac{B^{2p\left( j_{1}-j_{2}\right) +\left( 2-\alpha \right) j_{1}}}{%
4\pi }\left\{ \sum_{m=1}^{2p}[2\binom{2p}{m-1}+\binom{2p}{m}]B^{-\left(
2p+1-m\right) j_{1}}\frac{d^{m}}{d\omega ^{m}}\right\} \widehat{G}_{\alpha
,j_{1}j_{2}}\left( \omega \right) \left| _{\omega =B^{j_{1}}\left( 2\pi \mu
\mp \varphi \right) }\right. \\
&&+\frac{B^{2p\left( j_{1}-j_{2}\right) +\left( 2-\alpha \right)
j_{1}}e^{\pm i\varphi }}{4\pi }\left\{ B^{-\left( 2p+1\right) j_{1}}+2\frac{%
d^{2p+1}}{d\omega ^{2p+1}}\right\} \widehat{G}_{\alpha ,j_{1}j_{2}}\left(
\omega \right) \left| _{\omega =B^{j_{1}}\left( 2\pi \mu \mp \varphi \right)
}\right. \! \\
&\leq &\frac{2^{2p}C_{\alpha ,g}B^{-2p\left( j_{1}+j_{2}\right) +\log
_{B}j_{1}}}{\left( 2\pi \mu \mp \varphi \right) ^{4p+2-\alpha }},\text{ }\mu
=1,2,...
\end{eqnarray*}%
Therefore 
\begin{equation*}
\left| C_{B,g_{j_{1}},j_{2}}\right| \leq 2^{2p}C_{\alpha ,g}\left( \frac{1}{%
2\varphi ^{4p+2-\alpha }}+\left| 4p+1-\alpha \right| \pi ^{\alpha
-4p-1}\right) B^{-2p\left( j_{1}+j_{2}\right) +\log _{B}j_{1}}.
\end{equation*}%
It is then straightforward to conclude as in the case where $j_{1}=j_{2},$
to obtain%
\begin{eqnarray*}
Corr\left( \beta _{j_{1},k},\beta _{j_{2},k^{\prime }}\right) &\leq &C\frac{%
\theta ^{\alpha -4p-2}B^{-2p\left( j_{1}+j_{2}\right) +\log _{B}j_{1}}}{%
B^{\left( 1-\alpha /2\right) \left( j_{1}+j_{2}\right) }} \\
&\leq &C\theta ^{\alpha -4p-2}B^{-\left[ \left( j_{1}+j_{2}\right) /2-\log
_{B}\left( j_{1}+j_{2}\right) /2\right] \left( 4p+2-\alpha \right)
}\rightarrow 0,\text{ as }j_{2}\rightarrow \infty \text{ }.
\end{eqnarray*}%
Thus (\ref{3.1}) is established.
\end{proof}

\begin{remark}
By careful manipulation, we can obtain an explicit expression for the
constant $C_{M}$ in (\ref{3.1}), i.e. 
\begin{equation*}
C_{M}=2^{2p}\pi ^{M+1}\left( 4p+2-\alpha \right) ^{2}\Gamma \left(
4p+1-\alpha \right) c_{0}C_{g}\log B.\text{ }
\end{equation*}
\end{remark}

The previous result shows that Mexican needlets can enjoy the same
uncorrelation properties as standard needlets, in the circumstances where
the angular power spectrum is decaying ``slowly enough''. The extra log term
in (\ref{3.1}) is a consequence of a standard technical difficulty when
dealing with a boundary case in the integral in (\ref{intlog}).

\ \newline

\begin{remark}
To obtain central limit results for finite-dimensional statistics based on
nonlinear transformations of the Mexican needlets coefficients, it would be
sufficient to consider the case where $j=j^{\prime }.$ The asymptotic
uncorrelation we established in Theorem $\ref{uno}$ is stronger; indeed, for
many applications it is useful to consider different scales $\left\{
j\right\} $ at the same time. Because of this, it is also important to focus
on the correlation of Mexican needlet coefficients at different $j,j^{\prime
}.$ We stress that the need for such analysis was much more limited for
standard needlets; indeed in the latter case, given the compactly supported
kernel $b(.),$ the frequency support of the various coefficients is
automatically disjoint when $\left| j-j^{\prime }\right| \geq 2$, whence
(for completely observed random fields) standard needlet coefficients can be
correlated only for $\left| j-j^{\prime }\right| =1$.
\end{remark}

\section{Stochastic properties of Mexican needlet coefficients, II: lower
bounds}

In this Section, we complete the previous analysis, establishing indeed that
the random Mexican needlets coefficients are necessarily correlated at some
angular distance in the presence of faster memory decay. This is clearly
different from needlets, which are always uncorrelated. As mentioned in the
Introduction, the heuristic rationale behind this duality can be explained
as follows: it should be stressed that we are focussing on high-resolution
asymptotics, i.e. the asymptotic behaviour of random coefficients at smaller
and smaller scales in the same random realization. For such asymptotics, a
crucial role can be played by terms which remain constant across different
scales. In the case of usual needlets, which have bounded support over the
multipoles, terms like these are simply dropped by construction. This is not
so for Mexican needlets, which in any case include components at the lowest
scales. These components are dominant when the angular power spectrum decays
fast, and as such they prevent the possibility of asymptotic uncorrelation.
In particular, we have correlation when the angular power spectrum is such
that $\alpha >4p+2.$

\begin{theorem}
\label{due}Under condition \ref{cdA}, for $\alpha >4p+2$, $\forall $ $%
\varepsilon \in (0,1),$ there exists a positive $\delta \leq \varepsilon
\left( 1+c_{0}^{2}\right) ^{-1/\left( \alpha -4p-2\right) }$ such that 
\begin{equation}
\lim_{j\rightarrow \infty }\inf Corr\left( \beta _{jk;p},\beta _{jk^{\prime
};p}\right) >1-\varepsilon \text{ ,}  \label{3.2}
\end{equation}%
for all $\left\{ \xi _{jk},\xi _{jk^{\prime }}\right\} $ such that $d(\xi
_{jk},\xi _{jk^{\prime }})\leq \delta $ .
\end{theorem}

\begin{proof}
We first divide the variance of the coefficients into three parts, as follows%
\begin{eqnarray*}
&&\left( \underset{1\leq l<\epsilon _{1}B^{j}}{\sum }+\underset{\epsilon
_{1}B^{j}\leq l<\epsilon _{2}B^{j}}{\sum }+\underset{l\geq \epsilon _{2}B^{j}%
}{\sum }\right) \left( \frac{l\left( l+1\right) }{B^{2j}}\right)
^{2p}e^{-2l\left( l+1\right) /B^{2j}}\frac{2l+1}{4\pi }C_{l} \\
&=&:A_{1j}+A_{2j}+A_{3j}\text{ }.
\end{eqnarray*}%
Intuitively, our idea is to show that the first sum is of exact order $%
O(B^{\left( 2-\alpha \right) j}\times B^{\left( \alpha -4p-2\right)
j})=O(B^{-4pj}),$ while the second two are smaller ($O(B^{\left( 2-\alpha
\right) j})$ and $o(B^{\left( 2-\alpha \right) j})$, respectively). Indeed,
for the first part we obtain easily%
\begin{eqnarray*}
A_{1j} &=&\underset{1\leq l<\epsilon _{1}B^{j}}{\sum }\left( \frac{l\left(
l+1\right) }{B^{2j}}\right) ^{2p}e^{-2l\left( l-1\right) /B^{2j}}\frac{2l+1}{%
4\pi }C_{l}\leq \underset{1\leq l<\epsilon _{1}B^{j}}{2\sum }\frac{l^{4p}}{%
B^{4pj}}\frac{l}{\pi }l^{-\alpha }g_{j}(\frac{l}{B^{j}}) \\
&\leq &2\frac{c_{0}}{\pi }B^{\left( 2-\alpha \right)
j}\int_{B^{-j}}^{\epsilon _{1}}x^{4p+1-\alpha }dx=2\frac{c_{0}\left(
B^{\left( \alpha -4p-2\right) j}-\epsilon _{1}^{4p+2-\alpha }\right) }{\pi
\left( \alpha -4p-2\right) }B^{\left( 2-\alpha \right) j},
\end{eqnarray*}%
and%
\begin{equation*}
\underset{1\leq l<\epsilon _{1}B^{j}}{\sum }\left( \frac{l\left( l+1\right) 
}{B^{2j}}\right) ^{2p}e^{-2l\left( l+1\right) /B^{2j}}\frac{2l+1}{4\pi }%
C_{l}\geq \frac{\left( B^{\left( \alpha -4p-2\right) j}-\epsilon
_{1}^{4p+2-\alpha }\right) }{2\pi c_{0}\left( \alpha -4p-2\right) }%
e^{-2\epsilon _{1}^{2}}B^{\left( 2-\alpha \right) j}.
\end{equation*}%
Similarly, for the second part%
\begin{eqnarray*}
\frac{\left( \epsilon _{2}^{4p+2-\alpha }-\epsilon _{1}^{4p+2-\alpha
}\right) }{2\pi \left( \alpha -4p-2\right) c_{0}}e^{-2\epsilon
_{2}^{2}}B^{\left( 2-\alpha \right) j} &\leq &\underset{\epsilon
_{1}B^{j}\leq l<\epsilon _{2}B^{j}}{\sum }\left( \frac{l\left( l+1\right) }{%
B^{2j}}\right) ^{2p}e^{-2l\left( l+1\right) /B^{2j}}\frac{2l+1}{4\pi }C_{l}
\\
&\leq &2\frac{\left( \epsilon _{1}^{4p+2-\alpha }-\epsilon _{2}^{4p+2-\alpha
}\right) }{\pi \left( \alpha -4p-2\right) }c_{0}e^{-2\epsilon
_{1}^{2}}B^{\left( 2-\alpha \right) j}\text{ ,}
\end{eqnarray*}%
and for the third part%
\begin{eqnarray*}
&&\underset{l\geq \epsilon _{2}B^{j}}{\sum }\left( \frac{l\left( l+1\right) 
}{B^{2j}}\right) ^{2p}e^{-2l\left( l+1\right) /B^{2j}}\frac{2l+1}{4\pi }%
C_{l}\leq \frac{3c_{0}}{4\pi }B^{\left( 2-\alpha \right) j}\int_{\epsilon
_{2}}^{\infty }x^{4p+1-\alpha }e^{-x^{2}}dx \\
&\leq &\frac{3c_{0}}{4\pi }B^{\left( 2-\alpha \right) j}\epsilon
_{2}^{4p+1-\alpha }\int_{\epsilon _{2}}^{\infty }e^{-x^{2}}dx\leq \frac{%
3c_{0}}{4\pi }\epsilon _{2}^{4p-\alpha }e^{-\epsilon _{2}^{2}}B^{\left(
2-\alpha \right) j},\text{ }\left( \epsilon _{2}>1\right) \text{ }.
\end{eqnarray*}%
The last inequality follows from the asymptotic formula%
\begin{equation*}
\int_{y}^{\infty }e^{-x^{2}/2}dx\thicksim \frac{1}{y}e^{-y^{2}/2},\text{ \ }%
y\rightarrow \infty \text{ }.
\end{equation*}%
We have then that%
\begin{eqnarray*}
\frac{A_{3j}}{A_{1j}} &=&\left\{ \underset{1\leq l<\epsilon _{1}B^{j}}{\sum }%
\left( \frac{l\left( l+1\right) }{B^{2j}}\right) ^{2p}e^{-2l(l+1)/B^{2j}}%
\frac{2l+1}{4\pi }C_{l}\right\} ^{-1} \\
&&\times \left\{ \underset{l\geq \epsilon _{2}B^{j}}{\sum }\left( \frac{%
l\left( l+1\right) }{B^{2j}}\right) ^{2p}e^{-2l(l+1)/B^{2j}}\frac{2l+1}{4\pi 
}C_{l}\right\} 
\end{eqnarray*}%
\begin{equation*}
=O(B^{j(4p+2-\alpha )})=o(1)\text{ },\text{ }as\text{ }j\rightarrow \infty 
\text{ .}
\end{equation*}%
On the other hand, for any positive $\varepsilon <1,$ if we choose $\epsilon
_{1}=NB^{-j},$ where $N$ is sufficiently large that $N^{4p+2-\alpha }\left(
1+\frac{2}{\varepsilon }c_{0}^{2}\right) <1,$ we obtain that 
\begin{equation*}
\frac{\left( B^{\left( \alpha -4p-2\right) j}-\epsilon _{1}^{4p+2-\alpha
}\right) }{c_{0}}>\frac{2}{\varepsilon }\left( \epsilon _{1}^{4p+2-\alpha
}-\epsilon _{2}^{4p+2-\alpha }\right) c_{0}\text{ },\text{ }
\end{equation*}%
whence 
\begin{equation*}
\frac{A_{2j}}{A_{1j}}<\frac{\varepsilon }{2}\text{ as }j\rightarrow \infty 
\text{ }.
\end{equation*}%
Thus we obtain that%
\begin{eqnarray*}
\left\{ \underset{l\geq 1}{\sum }\left( \frac{l\left( l+1\right) }{B^{2j}}%
\right) ^{2p}e^{-2l(l+1)/B^{2j}}\frac{2l+1}{4\pi }C_{l}\right\} ^{-1} &=&%
\frac{1}{A_{1j}}\left\{ 1+\frac{A_{2j}}{A_{1j}}+\frac{A_{3j}}{A_{1j}}%
\right\} ^{-1} \\
&\geq &\frac{1}{A_{1j}}\left\{ 1+\frac{\varepsilon }{2}+o(1)\right\} ^{-1} \\
&\geq &cB^{-4pj},\text{ some }c>0\text{ .}
\end{eqnarray*}%
More explicitly, the variance at the denominator has the same order as for
the summation restricted to the elements in the range $1\leq l<\epsilon
_{1}B^{j}.$

To analyze the numerator, we start by recalling that%
\begin{equation*}
\sup_{\theta \in \lbrack 0,\pi ]}P_{l}(\cos \theta )=P_{l}(\cos 0)=1\text{ },%
\text{ and }\sup_{\theta \in \lbrack 0,\pi ]}|\frac{d}{d\theta }P_{l}(\cos
\theta )|\leq 3l\text{ }.
\end{equation*}%
As a consequence, for any $\varepsilon >0$ there exists a $\delta >0,$ s.t.
if $0<\theta \leq \delta \leq \varepsilon /6N,$ then,%
\begin{equation*}
\left| P_{l}(\cos \theta )-P_{l}(\cos 0)\right| \leq 3l\theta \leq
\varepsilon \text{ ,}
\end{equation*}%
for all $l>N,$ where the above inequalities follow from 
\begin{equation*}
0\leq \cos \left( \theta _{0}+\theta \right) -\cos \theta _{0}=2\sin ^{2}%
\frac{\theta }{2}\leq \theta \text{ , }\forall \text{ }\theta _{0}\in
\lbrack 0,2\pi )\text{ .}
\end{equation*}%
Therefore, for any $\xi _{jk},\xi _{jk^{\prime }}\in S^{2}$ s.t. $\arccos
\left\langle \xi _{jk},\xi _{jk^{\prime }}\right\rangle \leq \delta ,$ we
have 
\begin{equation*}
Corr\left( \beta _{jk},\beta _{jk^{\prime }}\right) =\frac{\underset{l\geq 1}%
{\sum }\left( \frac{l\left( l+1\right) }{B^{2j}}\right)
^{2p}e^{-2l(l+1)/B^{2j}}\frac{2l+1}{4\pi }C_{l}P_{l}\left( \left\langle \xi
_{jk},\xi _{jk^{\prime }}\right\rangle \right) }{\underset{l\geq 1}{\sum }%
\left( \frac{l\left( l+1\right) }{B^{2j}}\right) ^{2p}e^{-2l(l+1)/B^{2j}}%
\frac{2l+1}{4\pi }C_{l}}
\end{equation*}%
\begin{eqnarray*}
&\geq &\frac{\underset{1\leq l\leq N}{\sum }\left( \frac{l\left( l+1\right) 
}{B^{2j}}\right) ^{2p}e^{-2l(l+1)/B^{2j}}C_{l}P_{l}(\cos 0)}{\underset{l\geq
1}{\sum }\left( \frac{l\left( l+1\right) }{B^{2j}}\right)
^{2p}e^{-2l(l+1)/B^{2j}}\frac{2l+1}{4\pi }C_{l}} \\
&&-\frac{\underset{1\leq l\leq N}{\sum }\left( \frac{l\left( l+1\right) }{%
B^{2j}}\right) ^{2p}e^{-2l(l+1)/B^{2j}}\frac{2l+1}{4\pi }C_{l}\left|
P_{l}\left( \left\langle \xi _{jk},\xi _{jk^{\prime }}\right\rangle \right)
-P_{l}(\cos 0)\right| }{\underset{l\geq 1}{\sum }\left( \frac{l\left(
l+1\right) }{B^{2j}}\right) ^{2p}e^{-2l(l+1)/B^{2j}}\frac{2l+1}{4\pi }C_{l}}
\\
&&+\frac{\underset{l>N}{\sum }\left( \frac{l\left( l+1\right) }{B^{2j}}%
\right) ^{2p}e^{-2l(l+1)/B^{2j}}C_{l}\left| P_{l}\left( \left\langle \xi
_{jk},\xi _{jk^{\prime }}\right\rangle \right) \right| }{\underset{l\geq 1}{%
\sum }\left( \frac{l\left( l+1\right) }{B^{2j}}\right)
^{2p}e^{-2l(l+1)/B^{2j}}\frac{2l+1}{4\pi }C_{l}}
\end{eqnarray*}%
\begin{eqnarray*}
&\geq &\frac{\underset{1\leq l\leq N}{\sum }\left( \frac{l\left( l+1\right) 
}{B^{2j}}\right) ^{2p}e^{-2l(l+1)/B^{2j}}\frac{2l+1}{4\pi }C_{l}\times
(1-\varepsilon )}{\underset{l\geq 1}{\sum }\left( \frac{l\left( l+1\right) }{%
B^{2j}}\right) ^{2p}e^{-2l(l+1)/B^{2j}}\frac{2l+1}{4\pi }C_{l}}%
+O(B^{j(4p+2-\alpha )}) \\
&=&\frac{(1-\varepsilon )}{1+\varepsilon /2}+o\left( 1\right)
=(1-\varepsilon ^{\prime })+o\left( 1\right) \text{ },\text{ as }%
j\rightarrow \infty \text{ },\text{ }\varepsilon ^{\prime }=\frac{%
3\varepsilon }{2+\varepsilon }\text{ .}
\end{eqnarray*}%
Thus the proof is completed.
\end{proof}

\ 

As mentioned earlier, the results in the previous two Theorems illustrate an
interesting trade-off between the localization and correlation properties of
spherical needlets. In particular, we can always achieve uncorrelation by
choosing $p>(\alpha -2)/4$; of course $\alpha $ is generally unknown and
must be estimated from the data (in this sense standard needlets have better
robustness properties). Introducing higher order terms implies lowering the
weight of the lowest multipoles, i.e. improving the localization properties
in frequency space. On the other hand, it may be expected that such an
improvement of the localization properties in the frequency domain will lead
to a worsening of the localization in pixel space, as a consequence of the
Uncertainty Principle we mentioned in the Introduction (see for instance %
\cite{havin}). We do not investigate this issue here, but we shall provide
some numerical evidence on this phenomenon in an ongoing work.

\ \newline

\begin{remark}
In Theorem $\ref{due}$, we decided to keep the assumptions as close as
possible to Theorem \ref{uno}, in order to ease comparisons and highlight
the symmetry between the two results. However, it is simple to show that the
correlation result holds in much greater generality, for angular power
spectra that have a decay which is faster than polynomial. In particular,
assume that%
\begin{equation*}
C_{l}=H(l)\exp (-l^{p})\text{ , }l=1,2,...
\end{equation*}%
where $H(l)$ is any kind of polynomial such that $H(l)>c>0$ and $p>0.$ Then
it is simple to establish the same result as in Theorem \ref{due}, by means
of a simplified version of the same argument. The underlying rationale
should be easy to get: for exponentially decaying power spectra the
dominating components are at the lowest frequencies, and they introduce
correlations among all random coefficients which cannot be neglected.
\end{remark}

\section{Statistical Applications}

The previous results lend themselves to several applications for the
statistical analysis of spherical random fields, in particular with a view
to CMB\ data analysis. Similarly to \cite{bkmpb}, let us consider
polynomials functions of the normalized Mexican needlets coefficients, as
follows%
\begin{equation*}
h_{u,N_{j}}:=\frac{1}{\sqrt{N_{j}}}\sum_{k=1}^{N_{j}}%
\sum_{q=1}^{Q}w_{uq}H_{q}(\widehat{\beta }_{jk;p})\text{ , }\widehat{\beta }%
_{jk;p}:=\frac{\beta _{jk;p}}{\sqrt{E\beta _{jk;p}^{2}}}\text{ , }u=1,...,U%
\text{ ,}
\end{equation*}%
where $H_{q}(.)$ denotes the $q$-th order Hermite polynomials (see \cite{s}%
), $N_{j}$ is the cardinality of coefficients corresponding to frequency $j$
(where we take $\left\{ \xi _{jk}\right\} $ to form a $B^{-j}$-mesh, see %
\cite{bkmp}, so that $N_{j}\approx B^{2j})$, and $\left\{ w_{uq}\right\} $
is a set of deterministic weights that must ensure these statistics are
asymptotically nondegenerate, i.e.

\negthinspace\ \newline

\begin{condition}
\label{invert} There exist $j_{0}$ such that for all $j>j_{0}$%
\begin{equation*}
rank(\Omega _{j})=U\text{ , }\Omega _{j}:=Eh_{N_{j}}h_{N_{j}}^{\prime }\text{
, }h_{N_{j}}:=(h_{1,N_{j}},...,h_{U,N_{j}})^{\prime }\text{ .}
\end{equation*}
\end{condition}

\ 

Condition ($\ref{invert})$ is a standard invertibility assumption which will
ensure our statistics are asymptotically nondegenerate (for instance, it
rules out multicollinearity). Several examples of relevant polynomials are
given in \cite{bkmpb}; for instance, given a theoretical model for the
angular power spectrum $\left\{ C_{l}\right\} ,$ it is suggested in that
reference that a goodness-of-fit statistic might be based upon%
\begin{eqnarray*}
\frac{1}{\sqrt{N_{j}}}\sum_{k=1}^{N_{j}}H_{2}(\widehat{\beta }_{jk;p}) &=&%
\frac{1}{\sqrt{N_{j}}}\sum_{k=1}^{N_{j}}(\widehat{\beta }_{jk;p}^{2}-1) \\
&=&\frac{1}{\sqrt{N_{j}}}\sum_{k=1}^{N_{j}}(\frac{\beta _{jk;p}^{2}}{\lambda
_{jk}\underset{l\geq 1}{\sum }b^{2}(\frac{l}{B^{j}})\frac{2l+1}{4\pi }C_{l}}%
-1)\text{ .}
\end{eqnarray*}%
The statistic%
\begin{equation*}
\widehat{\Gamma }_{j}=\frac{1}{N_{j}}\sum_{k=1}^{N_{j}}\frac{\beta
_{jk;p}^{2}}{\lambda _{jk}}
\end{equation*}%
can then be viewed as an unbiased estimator for 
\begin{equation*}
\Gamma _{j}=E\widehat{\Gamma }_{j}=\underset{l\geq 1}{\sum }b^{2}(\frac{l}{%
B^{j}})\frac{2l+1}{4\pi }C_{l}\text{ .}
\end{equation*}%
We refer to \cite{fay08}, \cite{dela08} for the analysis of this estimator
in the presence of missing observations and noise, and for its application
to CMB data in the standard needlets case. Our results below can be viewed
as providing consistency and asymptotic Gaussianity (for fully observed maps
and without noise) in the Mexican needlets approach. As always in this
framework, consistency has a non-standard meaning, as we do not have
convergence to a fixed parameter, but rather convergence to unity of the
ratio $\widehat{\Gamma }_{j}/\Gamma _{j}$.

Likewise, tests of Gaussianity could be implemented by focussing on the
skewness and kurtosis of the wavelets coefficients (see for instance \cite%
{jfaa2}), i.e. by focussing on%
\begin{eqnarray*}
\frac{1}{\sqrt{N_{j}}}\sum_{k=1}^{N_{j}}\left\{ H_{3}(\widehat{\beta }%
_{jk;p})+H_{1}(\widehat{\beta }_{jk;p})\right\} &=&\frac{1}{\sqrt{N_{j}}}%
\sum_{k=1}^{N_{j}}\widehat{\beta }_{jk;p}^{3}\text{ and } \\
\frac{1}{\sqrt{N_{j}}}\sum_{k=1}^{N_{j}}\left\{ H_{4}(\widehat{\beta }%
_{jk;p})+6H_{2}(\widehat{\beta }_{jk;p})\right\} &=&\frac{1}{\sqrt{N_{j}}}%
\sum_{k=1}^{N_{j}}\left\{ \widehat{\beta }_{jk;p}^{4}-3\right\} \text{ .}
\end{eqnarray*}%
The joint distribution for these statistics is provided by the following
results:

\ \newline

\begin{theorem}
\label{statapp} Assume $T$ is a Gaussian mean square continuous and
isotropic random field; assume also that Conditions \ref{cdA}, \ref{invert}
are satisfied and choose $p>(\alpha +\delta )/4,$ some $\delta >0.$ Then as $%
N_{j}\rightarrow \infty $%
\begin{equation*}
\Omega _{j}^{-1/2}h_{N_{j}}\rightarrow _{d}N(0,I_{U})\text{ .}
\end{equation*}
\end{theorem}

\begin{proof}
The asymptotic behaviour of our polynomial statistics can be established by
means of the method of moments. In particular, it is possible to exploit the
diagram formula for higher order moments of Hermite polynomial, as explained
for instance in \cite{np},\cite{s}. The details are the same as in \cite%
{bkmpb}, and thus they are omitted for brevity's sake. We only note that, in
order to be able to use Lemma 6 in that reference, we need to ensure that%
\begin{equation*}
\left\vert Corr\left( \beta _{jk;p},\beta _{jk^{\prime };p}\right)
\right\vert \leq \frac{C}{\left( 1+B^{j}d\left( \xi _{jk},\xi _{jk^{\prime
}}\right) \right) ^{2+\delta }}\text{ , some }C>0\text{ .}
\end{equation*}%
In view of (\ref{3.1}), this motivates the tighter limit we need to impose
on the value of $p.$
\end{proof}

\negthinspace

It may be noted that the covariance matrix $\Omega _{j}$ can itself be
consistently estimated from the data at any level $j,$ for instance by means
of the bootstrap/subsampling techniques that are detailed in \cite{bkmp}.
Again, the arguments of that paper lend themselves to straightforward
extensions to the present circumstances, as they rely uniquely upon the
covariance inequalities for the random wavelet coefficients. Likewise, the
previous results may be extended to cover statistical functionals over
different frequencies, for instance the bispectrum (\cite{ejslan,m2007}). A
much more challenging issue relates to the relaxation of the Gaussianity
assumptions, which is still under investigation.

\newpage

Corresponding Author:

Domenico Marinucci

Dipartimento di Matematica

Universita' di Roma Tor Vergata

via della Ricerca Scientifica 1

00133 Italy

email: marinucc@mat.uniroma2.it

tel. +39-0672594105

fax +39-0672594699

\end{document}